\documentclass[11pt]{amsart}

\usepackage[margin=1in]{geometry}
\usepackage{amssymb,amsmath,enumerate,tikz}
\usepackage{graphicx}

\numberwithin{equation}{section}
\newtheorem{theorem}[equation]{Theorem}
\newtheorem{lemma}[equation]{Lemma}
\newtheorem{proposition}[equation]{Proposition}

\newtheorem{conjecture}[equation]{Conjecture}

\theoremstyle{definition}

\newtheorem*{remark}{Remark}

\def\E{\mathsf{E}}
\def\N{\mathsf{N}}
\def\P{\mathcal{P}}
\def\R{\mathcal{R}}
\def\Z{\mathbb{Z}}
\def\Grid{\mathcal{G}}
\def\vsp{\rule[-1ex]{0ex}{3.5ex}}

\colorlet{sMax}{yellow!85!orange}
\colorlet{2Max}{green!60!black}

\title{Points of maximal traffic on a grid with obstruction}
\author[Gil]{Juan Gil}
\address{Penn State Altoona\\ 3000 Ivyside Park\\ Altoona, PA 16601}
\author[Liang]{Zhenni Liang}
\address{University of North Texas}
\author[Odetola]{Ayodeji Odetola}
\author[Weiner]{Michael Weiner}
\address{Penn State Altoona\\ 3000 Ivyside Park\\ Altoona, PA 16601}

\begin{document}
\begin{abstract}
For $n\in\mathbb{N}$, we consider the set of lattice paths from $(0,0)$ to $(n,n)$ using only unit north and east steps. Given a point $B$ to be avoided, we ask: at which point $A$ on the grid with corners $(0,0)$ and $(n,n)$, different from the endpoints, does the largest number of $B$-avoiding lattice paths pass through? We show that for $n\ge 9$, regardless of the location of $B$, the maximum is attained at one of ten specific points clustered near the two endpoints of the grid. This stability, however, conceals an interesting anomaly. When the obstruction $B$ lies on the antidiagonal $x+y=n$, the points of maximal traffic migrate from the near-corner points $(1,1)$ and $(n-1,n-1)$ to boundary points in the set of possible maximizers. The migration occurs for every $8\le n\le 375$, and intermittently up to $n=495$. We conjecture that the anomaly disappears for $n\ge 496$.
\end{abstract}

\maketitle

\section{Introduction}

In~\cite{K2017}, Kaplan posed the following question. Suppose everyone in town lives at $(0,0)$ and works at $(n,n)$, and walks to work along a uniformly random monotone lattice path (using only unit north and east steps). Where should Nathan open his restaurant to maximize the chance that a commuter will pass by? Kaplan showed that, for $n\ge 2$, the answer is $(1,1)$ and $(n-1,n-1)$, with $(1,0)$, $(0,1)$, $(n-1,n)$, and $(n,n-1)$ giving the next-best options. 

We essentially pick up where Kaplan left off. The restaurant location problem assumes a clean grid; in practice, however, parts of the city are often inaccessible. Suppose a road closure or construction site occupies some lattice point $B=(c,d)$ with $0\le c,d\le n$ and $0<c+d<2n$, forcing commuters to detour around it. Where should a food truck park to catch the most foot traffic? More precisely, among the lattice paths from $(0,0)$ to $(n,n)$ that do \emph{not} pass through $B$, which point $A$ carries the maximum traffic? For example, when $n=3$ and $B=(1,1)$, there are seven different points on the grid where the maximum happens. On the other hand, when $n=4$ and $B=(1,2)$, there is a unique maximum at the point $(2,1)$, see Table~\ref{tab:n=3&4}.

While small cases reveal a sensitivity to the location of $B$, the answer stabilizes for large values of $n$. In this paper, we prove that for $n\ge 9$, the maximum is always attained at one of the following ten points: 
\begin{gather*}
 (1,0),\ (0,1),\ (1,1),\ (2,1),\ (1,2), \\[2pt]
 (n-1,n),\ (n,n-1),\ (n-1,n-1),\ (n-1,n-2),\ (n-2,n-1),
\end{gather*}
regardless of where the forbidden point $B$ is located. Moreover, if the obstruction is in the set $\{(1,2), (2,1)\}$, then the point of maximum traffic is the other point of that set. By symmetry, an analogous statement holds for the set $\{(n-1,n-2), (n-2,n-1)\}$. For every other location of the obstruction, the maximum is attained at one of the remaining six points. The bound $n\ge 9$ is sharp: for $n=8$, there are four more locations where the maximum can occur.

In other words, for $n\ge 9$, the points in $\mathcal{S}=\{(1,2),(2,1),(n-1,n-2),(n-2,n-1)\}$ are the only obstructions for which the maximum leaves the six corner-adjacent points.

Another surprising behavior worth emphasizing occurs when the obstruction sits on the antidiagonal $x+y=n$. There $B$ is equidistant from the two endpoints of the grid, and the competition among the six corner-adjacent candidates degenerates: the reflection about the antidiagonal fixes $B$, forcing the exact ties $f_B(1,1)=f_B(n-1,n-1)$, $f_B(1,0)=f_B(n,n-1)$, and $f_B(0,1)=f_B(n-1,n)$. Locating the maximum therefore reduces to a single comparison between $f_B(1,1)$ and one boundary point (namely $f_B(1,0)$ when $B$ lies above the main diagonal, and $f_B(0,1)$ when it lies below). On the grid without obstructions, this comparison is decided by a tiny margin: $f(1,1)-f(1,0)$ is only a $\frac{1}{2n-1}$ fraction of $f(1,0)$. However, an obstruction on the antidiagonal penalizes $(1,1)$ more heavily than $(1,0)$, and for suitably placed $B$, the differential penalty can exceed that vanishing margin, moving the maximum to a pair of boundary points in set of eligible maximizers. In Section~\ref{sec:antidiagonal} we make this precise, showing that the migration occurs exactly when an explicit ratio of binomial quantities exceeds 1. Evaluating that criterion reveals an irregular pattern: the migration happens for every $8\le n\le 375$, then intermittently up to $n=495$, and then seems to stop. We conjecture that the anomaly never reappears for $n\ge 496$.

For $n\in\mathbb N$, we let $\Grid(n)\subset \Z^2$ denote the $(n\times n)$-grid with corners at $(0,0)$ and $(n,n)$, and we let $\R_0$, $\R_1$, and $\R_2$ be the regions of $\Grid(n)$ defined by
\begin{align*}
 \R_0(n) &= \{(x,y)\in\Grid(n): x=y, \text{ and } 0<x<n\}, \\
 \R_1(n) &= \{(x,y)\in\Grid(n): |x-y| = 1\}, \\
 \R_2(n) &= \{(x,y)\in\Grid(n): |x-y|\ge 2\}.
\end{align*}
Clearly, $\Grid(n)\backslash\{(0,0),(n,n)\} = \R_0(n)\cup \R_1(n)\cup \R_2(n)$, see Figure~\ref{fig:zoning}. The proofs of our results rely on inequalities obtained according to the position of the obstruction $B$ relative to these regions. The arguments combine explicit binomial estimates with injective maps between path classes. 

\tikzstyle{zone}=[rounded corners, line width=5, line cap=round, opacity=0.12]

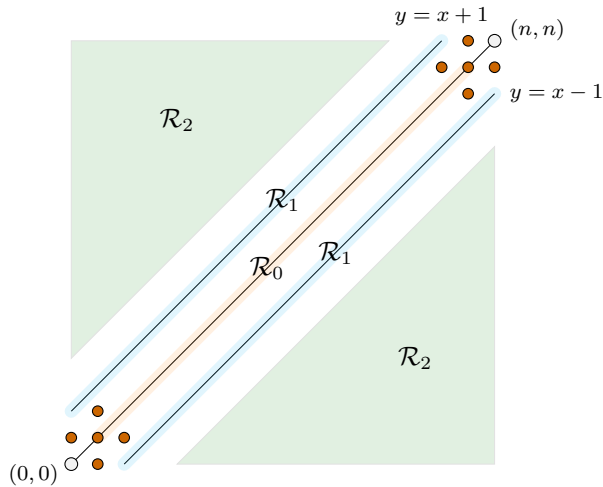
\begin{figure}[ht]
\begin{tikzpicture}[scale=0.7]
\draw[very thin] node[below=4, left=1]  {\scriptsize $(0,0)$} (0,0) -- (8,8)  node[above=4, right=2]  {\scriptsize $(n,n)$};
\foreach \i in {0,8}{\draw[fill=gray!10] (\i,\i) circle (0.12);}
\draw[zone, orange] (0.5,0.5) -- (7.5,7.5);
\draw[very thin] (1,0) -- (8,7) node[right=2] {\scriptsize $y=x-1$};
\draw[zone, cyan] (1,0) -- (8,7);
\draw[very thin] (0,1) -- (7,8) node[above=2] {\scriptsize $y=x+1$};
\draw[zone, cyan] (0,1) -- (7,8);
\draw[very thin, fill=green!50!black, opacity=0.12] (0,2) -- (6,8) -- (0,8) -- cycle;
\draw[very thin, fill=green!50!black, opacity=0.12] (2,0) -- (8,6) -- (8,0) -- cycle;
\foreach \x/\y in {0/1,1/0,1/1,1/2,2/1} { \draw[fill=orange!80!black] (\x/2,\y/2) circle (0.1);}
\foreach \x/\y in {7/8,8/7,7/7,6/7,7/6} { \draw[fill=orange!80!black] (4+\x/2,4+\y/2) circle (0.1);}
\node at (3.7,3.7) {\small $\R_0$};
\node at (4,5) {\small $\R_1$};
\node at (5,4) {\small $\R_1$};
\node at (2,6.5) {\small $\R_2$};
\node at (6.5,2) {\small $\R_2$};
\end{tikzpicture}
\caption{Zone partition and the 10 possible locations of the maximum.}
\label{fig:zoning}
\end{figure}

The paper is structured as follows. We start with a short section describing the background and discussing some basic properties. We also discuss the simple cases when the obstruction  $B$ is on the boundary of $\Grid(n)$. Section~\ref{sec:2unitsAway} deals with the case $B\in \R_2(n)$. In Section~\ref{sec:subdiagonals} we examine the case when $B$ is in $\R_1(n)$. This case is further split into cases depending on the proximity of $B$ to the points $(1,1)$ and $(n-1,n-1)$. The case when $B$ is on the diagonal of $\Grid(n)$ is handled in Section~\ref{sec:diagonal}. In Section~\ref{sec:comments} we summarize the resulting classification and record the small cases $5\le n\le 8$, which show how the stable picture is reached. Section~\ref{sec:antidiagonal} is devoted to the antidiagonal anomaly described above; we derive a closed-form criterion for the flip, state the resulting conjecture, and gather some open problems.

For background on lattice path enumeration techniques, including the inclusion-exclusion identities used throughout, we refer to the expository article by Krattenthaler in \cite{Kratt}.

\section{Background and basic cases} 

For $n\in\mathbb{N}$, let $\P_n$ be the set of paths from $(0,0)$ to $(n,n)$ using only $\E=(1,0)$ and $\N=(0,1)$ steps. For a point $A$ in $\Grid(n)$, we let
\[ f(A) = \#\{p\in\P_n: p \text{ passes through the point } A\}. \]
If $A=(a,b)$, then every $\E$-$\N$ path from $(0,0)$ to $(a,b)$ can be identified with a word of length $a+b$ over the alphabet $\{\E,\N\}$ having exactly $a$ copies of $\E$ and $b$ copies of $\N$. Similarly, a path from $(a,b)$ to $(n,n)$ corresponds to a word of length $2n-a-b$ having $n-a$ $\,\E$'s and $n-b$ $\,\N$'s. Therefore,
\[  f(A) = \binom{a+b}{a}\binom{2n-a-b}{n-a}. \]
In particular, $\big|\P_n\big| = f(0,0) = \binom{2n}{n}$. Kaplan \cite{K2017} showed that for $n\ge 2$ and $A\in\Grid(n)\backslash\{(0,0),(n,n)\}$, the maximum value of $f(A)$ is attained at the points $A=(1,1)$ and $A=(n-1,n-1)$. In other words,
\[ f(A)\le f(1,1) \;\text{ for every } A\in\Grid(n)\backslash\{(0,0),(n,n)\}, \]
with equality exactly when $A\in\{(1,1),(n-1,n-1)\}$.

The strategy is to optimize $f$ along the antidiagonals $a+b=k$ and then along the lines $a=b$ and $a=b+1$, reducing the two-variable problem to ratios of consecutive binomial coefficients. 

\medskip
As one can expect, the presence of an obstruction makes the problem more delicate. We now turn to the basic definitions and easy cases. For $B$ in $\Grid(n)\backslash\{(0,0),(n,n)\}$ and a list $L$ of points $A_1,\dots,A_m\in\Grid(n)$, we let
\begin{gather*} 
f(L) = \#\{p\in\P_n: p \text{ passes through all points in } L\}, \text{ and} \\
f_B(L) = \#\{p\in\P_n: p \text{ passes through every point in $L$, but avoids $B$} \}.
\end{gather*}
Let $A=(a,b)$ and $B=(c,d)$ be points in $\Grid(n)$ different from $(0,0)$ and $(n,n)$, and such that $a\le c$ and $b\le d$. Clearly, $f_B(A) = f(A) - f(A,B)$ and so
\begin{equation*}
 f_B(A) = \binom{a+b}{a}\left[\binom{2n-a-b}{n-a} - \binom{c+d-a-b}{c-a}\binom{2n-c-d}{n-c}\right].
\end{equation*}
From this formula we can easily verify the symmetric identities:
\begin{equation}\label{eq:symmetry}
\begin{aligned}
 f_B(A) &= f_{B'}(A'), \text{ where $A'=(b,a)$ and $B'=(d,c)$,} \\
 f_B(A) &= f_{\tilde B}(\tilde A), \text{ where $\tilde A=(n-a,n-b)$ and $\tilde B=(n-c,n-d)$}.
\end{aligned}
\end{equation}
Note that $A'$ is the reflection of $A$ about the diagonal $y=x$, and $\tilde A$ is the reflection of $A$ about the line $x+y=n$.

\begin{lemma}\label{lem:axes}
Let $A\in\Grid(n)\backslash\{(0,0),(n,n)\}$.
\begin{enumerate}[$(i)$]
\item If $B\in\{(0,1),(1,0),(n-1,n),(n,n-1)\}$, then $f_B(A)\le f_B(B')$.
\item If $B=(c,0)$ or $B=(0,c)$ with $1<c\le n$, then $f_B(A)\le f_B(1,1)$.
\item If $B=(d,n)$ or $B=(n,d)$ with $0\le d<n-1$, then $f_B(A)\le f_B(n-1,n-1)$.
\end{enumerate}
In $(ii)$ and $(iii)$, equality can occur only when $A\in\{(1,1),(n-1,n-1)\}$.
\end{lemma}

\begin{proof}
$(i)$ For $B\in\{(0,1),(1,0),(n-1,n),(n,n-1)\}$, every path through $A$ passes through exactly one of $B$ or $B'$, so $f(A)=f(A,B)+f(A,B')$. Hence $f_B(A)=f(A,B')\le f(B')=f_B(B')$.

$(ii)$ Since $B$ lies on an axis with $\|B\|>1$, no monotone path passes through both $(1,1)$ and $B$, so $f_B(1,1)=f(1,1)$. Therefore,
\[ f_B(A)\le f(A)\le f(1,1)=f_B(1,1), \]
where the second inequality is Kaplan's bound \cite{K2017}. The first inequality is strict unless $f(A,B)=0$, and the second is strict unless $A\in\{(1,1),(n-1,n-1)\}$; equality throughout therefore forces $A\in\{(1,1),(n-1,n-1)\}$.

$(iii)$ If $B=(d,n)$ with $d<n-1$, then $\tilde B=(n-d,0)$ with $n-d>1$, and $(ii)$ gives the inequality $f_{\tilde B}(\tilde A)\le f_{\tilde B}(1,1)$. By symmetry (identity \eqref{eq:symmetry}), this is $f_B(A)\le f_B(n-1,n-1)$. 

The case $B=(n,d)$ follows with a similar argument.
\end{proof}

When the point being avoided is not on the boundary of $\Grid(n)$, things are more interesting and less predictable. However, as we will show in later sections, for $n\ge 9$ the maximum values only happen at one or two of ten possible locations close to the endpoints of the grid. The case $n=2$ is trivial. Table~\ref{tab:n=3&4} shows a summary of the situation for $n=3$ and $n=4$ when the avoided point $B=(c,d)$ is such that $c\le d$. For example, if the point $(1,2)$ in $\Grid(3)$ is blocked, then $(2,1)$ is the point of maximal traffic with approximately 82\% of the $B$-avoiding paths going through it.

\begin{table}[ht]
\begin{tabular}{c|c|r|c|l}
 & $B$ & Max. & \%  & Location(s) of maximum traffic \\[2pt] \hline\hline
\vsp $n=3$ & $(1,1)$ & 5 & 50\% & $(0,1),(1,0),(0,2),(2,0),(2,2),(2,3),(3,2)$ \\
\vsp & $(1,2)$ & 9 & 82\% & $(2,1)$ \\
\vsp & $(2,2)$ & 5 & 50\% & $(0,1), (1,0), (1,3), (2,2), (2,3), (3,1), (3,2)$ \\[4pt] \hline\hline
\vsp $n=4$ & $(1,1)$ & 32 & 53\% & $(3,3)$ \\
\vsp & $(1,2)$ & 30 & 75\% & $(2,1)$ \\
\vsp & $(1,3)$ & 36 & 67\% & $(2,2)^\dagger$ \\
\vsp & $(2,2)$ & 20 & 50\% & $\{(0,1),(1,0),(3,4),(4,3)\}$ \\
\vsp & $(2,3)$ & 30 & 75\% & $(3,2)$ \\
\vsp & $(3,3)$ & 32 & 53\% & $(1,1)$ \\[4pt] \hline
\end{tabular}
\bigskip
\caption{Maximum-traffic locations for small $n$, with $B=(c,d)$ satisfying $c\le d$. The percentage is computed relative to the number of $B$-avoiding paths. The entry $\dagger$ is the unique location not among the ten special points. The remaining obstructions can be handled using Lemma~\ref{lem:axes} and the symmetry identities \eqref{eq:symmetry}.}
\label{tab:n=3&4}
\end{table}

For the remainder of this paper we will assume $n\ge 5$. 

\section{Obstructions in the region $\R_2$} 
\label{sec:2unitsAway}

Recall that $\R_2(n) = \{(x,y)\in\Grid(n): |x-y|\ge 2\}$.

\begin{lemma} \label{lem:thirdTotal}
If $(a,b)\in\R_2(n)$, then $f(a,b) < \frac13 \binom{2n}{n}$. In other words, the number of paths in $\P_n$ passing through $(a,b)$ is less than a third of the total number of paths in $\P_n$.
\end{lemma}
\begin{proof}
Assume without loss of generality that $b\ge a+2$. Then $(a,b)$, $(a+1,b-1)$, $(b,a)$ are 3 distinct points on the line $x+y = a+b$, and we have
\begin{equation*}
 \frac{f(a+1,b-1)}{f(a,b)} = \frac{\binom{a+b}{a+1}\binom{2n-a-b}{n-b+1}}{\binom{a+b}{b}\binom{2n-a-b}{n-a}} 
 = \frac{b(n-a)}{(a+1)(n-b+1)} > 1
\end{equation*}
since $b>a+1$ and $n-a>n-b+1$. Therefore, $f(a,b) < f(a+1,b-1)$. Finally, since $f(a,b)=f(b,a)$, we conclude $3f(a,b) < f(a,b)+f(a+1,b-1)+f(b,a) \le \tbinom{2n}{n}$.
\end{proof}

\begin{lemma} \label{lem:A_in_R2}
Let $B=(c,d)\in\R_2(n)$ be such that $c\ge 1$, $d\ge 1$, and $c+d\le n$. If $A\in\R_2(n)$, then $f_B(A) \le f_B(1,1)$.
\end{lemma}
\begin{proof}
Observe that $(c-1,d-1)$ is a point in $\Grid(n-1)$ with $|(d-1)-(c-1)| = |d-c|\ge 2$ and therefore, by Lemma~\ref{lem:thirdTotal},
\begin{equation*} 
  \binom{c-1+ d-1}{c-1}\binom{2(n-1)-(c-1)-(d-1)}{n-1-(c-1)} < \frac13 \binom{2(n-1)}{n-1}.
\end{equation*}
Now, writing the left-hand side of the inequality as $\binom{c+d-2}{c-1}\binom{2n-c-d}{n-c}$, we get
\begin{align*} 
 f_B(1,1) &= 2\binom{2n-2}{n-1} -2 \binom{c+d-2}{c-1}\binom{2n-c-d}{n-c} \\[1ex]
 &> 2\binom{2n-2}{n-1} - \frac23 \binom{2n-2}{n-1} = \frac43 \binom{2n-2}{n-1}. 
\end{align*}
On the other hand,
\[ f_B(A)\le f(A) < \frac13 \binom{2n}{n} = \frac{2(2n-1)}{3n} \binom{2n-2}{n-1} < \frac43 \binom{2n-2}{n-1}, \]
which implies $f_B(A)<f_B(1,1)$.
\end{proof}

\begin{lemma} \label{lem:subDiagonal}
Let $B=(c,d)\in\R_2(n)$ be such that $c\ge 1$, $d\ge 1$, and $c+d\le n$. Let $a$ be an integer such that $1\le a\le n$. 
\begin{enumerate}[$(i)$]
\item If $d\ge c+2$, then $f_B(a-1,a)\le f_B(a,a-1) \le f_B(1,0)$.
\item If $c\ge d+2$, then $f_B(a,a-1)\le f_B(a-1,a) \le f_B(0,1)$.
\end{enumerate}
\end{lemma}
\begin{proof}
We will only prove $(i)$. Part $(ii)$ follows with a similar symmetric argument. We consider three separate cases for the values of $a$ as illustrated in Figure~\ref{fig:subDiagonal-cases}. 

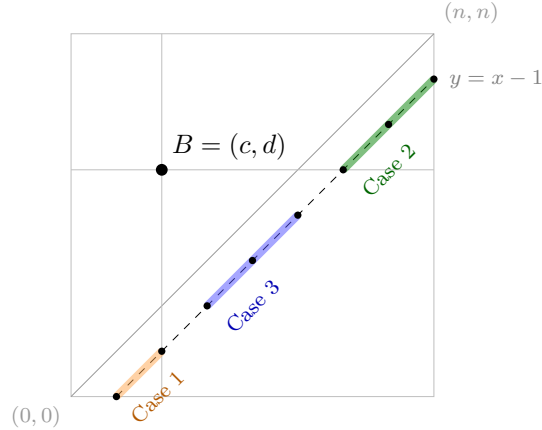
\begin{figure}[ht]
\begin{tikzpicture}[scale=0.6]
\draw[very thin, gray!50] (0,0) rectangle (8,8);
\draw[very thin, gray!80] node[below left] {\scriptsize $(0,0)$} (0,0) -- (8,8) node[above right] {\scriptsize $(n,n)$};
\draw[very thin, dashed] (1,0) -- (8,7);
\draw[very thin, gray!50] (2,0) -- (2,8);
\draw[very thin, gray!50] (0,5) -- (8,5);
\node[gray,right=2] at (8,7) {\scriptsize $y=x-1$};
\fill[black] (2,5) circle (0.13) node[above right] {\small $B=(c,d)$};
\draw[line width=3, orange!70, opacity=0.5] (1,0) -- (2,1);
\node[orange!70!black,rotate=45] at (1.9,0) {\scriptsize Case 1};
\draw[line width=3, blue!70, opacity=0.5] (3,2) -- (5,4);
\node[blue!70!black,rotate=45] at (4,2) {\scriptsize Case 3};
\draw[line width=3, green!50!black, opacity=0.5] (6,5) -- (8,7);
\node[green!40!black,rotate=45] at (7,5) {\scriptsize Case 2};
\foreach \x/\y in {1/0, 2/1, 3/2, 4/3, 5/4, 6/5, 7/6, 8/7}{
  \fill[black] (\x,\y) circle (0.08);
}
\end{tikzpicture}
\caption{The point $A=(a,a-1)$ moves along the subdiagonal $y=x-1$.}
\label{fig:subDiagonal-cases}
\end{figure}

\smallskip
\noindent{\em Case 1:} $a\le c$. Let $g_B(A)$ be the number of paths from $A$ to $(n,n)$ that avoid the point $B=(c,d)$. If $A_1=(a,a+1)$, $A_1'=(a+1,a)$, and $A_2=(a+2,a-1)$, then
\[ f_B(a,a-1) = \binom{2a-1}{a} \big[g_B(A_1) + 2g_B(A_1') + g_B(A_2)\big]. \]
Moreover, with $I=(1,0)$, we have
\begin{align*}
 f_B(I,A_1) &= \tbinom{2a}{a-1} g_B(A_1) \ge \tbinom{2a-1}{a} g_B(A_1), \\
 f_B(I,A_1') &= \tbinom{2a}{a} g_B(A_1') = 2\tbinom{2a-1}{a} g_B(A_1'), \\
 f_B(I,A_2) &= \tbinom{2a}{a-1} g_B(A_2) \ge \tbinom{2a-1}{a} g_B(A_2),
\end{align*}
and thus $f_B(a,a-1)\le f_B(I,A_1) + f_B(I,A_1') + f_B(I,A_2) \le f_B(I)$.

\medskip
Furthermore, since $\binom{c+d-2a+1}{c-a}\le \binom{c+d-2a+1}{c-a+1}$ for $d\ge c+1$, we have 
\[  \binom{2a-1}{a}\binom{c+d-2a+1}{c-a}\binom{2n-c-d}{n-c} \le \binom{2a-1}{a-1}\binom{c+d-2a+1}{c-a+1}\binom{2n-c-d}{n-c}. \]
Thus the number of paths through $(a,a-1)$ and $B$ is less than or equal to the number of paths through $(a-1,a)$ and $B$. Since $f(a-1,a)=f(a,a-1)$, we then get
\[ f(a-1,a) - f((a-1,a),B) \le  f(a,a-1) - f((a,a-1),B), \]
which means $f_B(a-1,a)\le f_B(a,a-1)$. This completes the proof of $(i)$ when $a\le c$.

\smallskip
\noindent{\em Case 2:} $a-1\ge d$. In this case, one can use a symmetric argument (looking at paths from $(n,n)$ to $(0,0)$) to conclude that $f_B(a-1,a)\le f_B(a,a-1)\le f_B(n,n-1)$. Since $f(n,n-1)=f(I)$, we have $f_B(I)-f_B(n,n-1) = f((n,n-1),B) - f(I,B)$. A direct computation gives
\[ \frac{f((n,n-1),B)}{f(I,B)} = \frac{(c+d)(n-d)}{c(2n-c-d)}, \]
and the inequality $(c+d)(n-d) \ge c(2n-c-d)$ simplifies to $(d-c)(n-c-d) \ge 0$, which holds because $d>c$ and $c+d\le n$. Hence $f_B(n,n-1) \le f_B(I)$, and so $f_B(a-1,a) \le f_B(a,a-1) \le f_B(I)$.

\smallskip
\noindent{\em Case 3:} $c<a<d+1$. Let $A_0=(a,a)$ and $A_3=(a+1,a-1)$. If $c<a<d$, then every path in $\P_n$ passes through at most one of the points $B$, $A_0$, or $A_3$, and so $f(I)$ is greater than the sum of $f(I,B)$, $f(I,A_0)$, and $f(I,A_3)$. Thus, 
\[ f_B(I) = f(I) - f(I,B) > f(I,A_0)+f(I,A_3). \]
Now, since $f(I,A_0) =\binom{2a-1}{a}\binom{2n-2a}{n-a}$ and $f(I,A_3) =\binom{2a-1}{a}\binom{2n-2a}{n-a-1}$, we get
\begin{equation} \label{eq:throughA}
 f(I,A_0)+f(I,A_3) = \binom{2a-1}{a}\binom{2n-2a+1}{n-a} = f(a,a-1) = f_B(a,a-1), 
\end{equation}
and so $f_B(I)>f_B(a,a-1)$.

If $a=d$, then there are paths that go through both $B$ and $A_0$, and the above decomposition is not disjoint. However, every path $p\in\P_n$ passing through $I$, $B$, and $A_0$ can be injectively mapped into a path $\phi(p)$ passing through $I$, $B_3= (c+1,d+1)$, and either $B_1=(c-1,d+1)$ or $B_2=(c+1,d-1)$. 

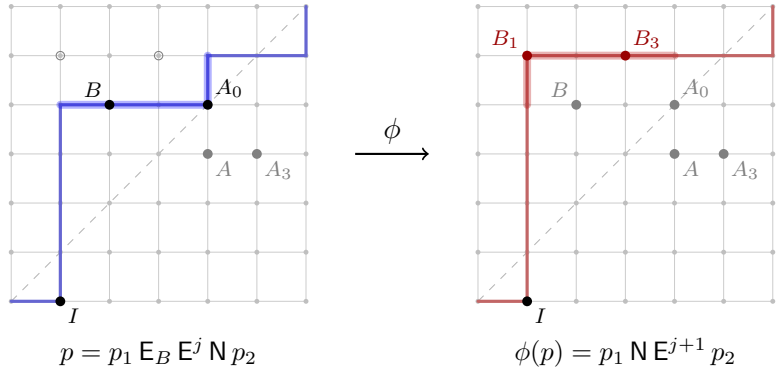
\begin{figure}[h]
\begin{tikzpicture}[scale=0.65]
\begin{scope}
\draw[very thin, gray!40] (0,0) grid (6,6);
\draw[very thin, dashed, gray!60] (0,0) -- (6,6);
\foreach \x in {0,...,6}{\foreach \y in {0,...,6}{\fill[gray!50] (\x,\y) circle (0.05);}}
\draw[very thick, line cap = round, blue!70!black!60] (0,0) -- (1,0) -- (1,4) -- (2,4) -- (4,4) -- (4,5) -- (6,5) -- (6,6);
\draw[line width = 3, line cap = round, blue, opacity=0.3]  (1,4) -- (4,4) -- (4,5);
\fill (1,0) circle (0.1) node[below right=-1pt] {\scriptsize $I$};
\fill (2,4) circle (0.1) node[above left=-1pt] {\scriptsize $B$};
\fill (4,4) circle (0.1) node[above right=-1pt] {\scriptsize $A_0$};
\fill[gray] (4,3) circle (0.1) node[below right=-1pt] {\scriptsize $A$};
\fill[gray] (5,3) circle (0.1) node[below right=-1pt] {\scriptsize $A_3$};
\draw[gray] (1,5) circle(0.085);
\draw[gray] (3,5) circle(0.085);
\node[below=10] at (3,0) {\small $p = p_1\, \mathsf{E}_B\, \mathsf{E}^j\, \mathsf{N}\, p_2$};
\end{scope}
\draw[->, thick] (7,3) -- (8.5,3) node[midway, above] {$\phi$};
\begin{scope}[xshift=9.5cm]
\draw[very thin, gray!40] (0,0) grid (6,6);
\draw[very thin, dashed, gray!60] (0,0) -- (6,6);
\foreach \x in {0,...,6}{\foreach \y in {0,...,6}{\fill[gray!50] (\x,\y) circle (0.05);}}
\draw[very thick, line cap = round, red!60!black!60] (0,0) -- (1,0) -- (1,5) -- (6,5) -- (6,6);
\draw[line width = 3, line cap = round, red!70!black, opacity=0.3]  (1,4) -- (1,5) -- (4,5);
\fill (1,0) circle (0.1) node[below right=-1pt] {\scriptsize $I$};
\fill[red!60!black] (1,5) circle (0.1) node[above left=-1pt] {\scriptsize $B_1$};
\fill[red!60!black] (3,5) circle (0.1) node[above right=-1pt] {\scriptsize $B_3$};
\fill[gray] (2,4) circle (0.1) node[above left=-1pt] {\scriptsize $B$};
\fill[gray] (4,4) circle (0.1) node[above right=-1pt] {\scriptsize $A_0$};
\fill[gray] (4,3) circle (0.1) node[below right=-1pt] {\scriptsize $A$};
\fill[gray] (5,3) circle (0.1) node[below right=-1pt] {\scriptsize $A_3$};
\node[below=10] at (3,0) {\small $\phi(p) = p_1\, \mathsf{N}\, \mathsf{E}^{j+1}\, p_2$};
\end{scope}
\end{tikzpicture}
\caption{Illustration of $\phi$. The path $\mathsf{E}_B \mathsf{E}^j \mathsf{N}$ is replaced by $\mathsf{N}\mathsf{E}^{j+1}$.}
\label{fig:phi-injection}
\end{figure}

More precisely, if $p$ is of the form $p=p_1\E_B\E^j\N\, p_2$, where $\E_B$ is an $\E$-step ending at $B$, then we let $\phi(p)=p_1\N\E^{j+1} p_2$. This path passes through $I$, $B_1$, and $B_3$. On the other hand, if $p=q_1\N_B\E^j\N\, p_2$, where $\N_B$ is an $\N$-step ending at $B$, then we let $\phi(p)=q_1\E\N^2\E^{j-1} p_2$, which passes through $I$, $B_2$, and $B_3$. In other words, $f(I,B,A_0) \le f_B(I,B_3)$, and therefore
\begin{align*}
 f_B(a,a-1) &= f(I,A_0)+f(I,A_3) \qquad \text{[by equation \eqref{eq:throughA}]} \\
 &=  f_B(I,A_0) + f(I,B,A_0) + f_B(I,A_3) \\
 &\le f_B(I,A_0) + f_B(I,B_3) + f_B(I,A_3) < f_B(I).
\end{align*}
In conclusion, if $c<a<d+1$, we have $f_B(1,0) = f_B(I) > f_B(a,a-1)$.
\end{proof}

\begin{lemma} \label{lem:Diagonal}
Let $B=(c,d)\in\R_2(n)$ be such that $c\ge 1$, $d\ge 1$, and $c+d\le n$. Then,
\begin{equation*}
 f_B(a,a) \le
 \begin{cases}
 f_B(1,1) &\text{if }\, 1\le a\le \min\{c,d\},\\
 f_B(n-1,n-1) &\text{if }\, \max\{c,d\}\le a \le n-1.
 \end{cases}
\end{equation*}
On the other hand, if $\min\{c,d\}<a< \max\{c,d\}$ and $n\ge 8$, then $f_B(a,a) \le \max\{f_B(1,0),f_B(0,1)\}$.
\end{lemma}

\begin{proof}
Without loss of generality, we assume $d\ge c+2$. As above, let $g_B(A)$ be the number of paths from $A$ to $(n,n)$ that avoid the point $B$. Let $A_1=(a,a+1)$ and $A_1'=(a+1,a)$.

If $a\le c$, then
\begin{align*}
 f_B((1,1),A_1) &= 2\tbinom{2a-1}{a} g_B(A_1) = \tbinom{2a}{a} g_B(A_1), \\
 f_B((1,1),A_1') &= 2\tbinom{2a-1}{a} g_B(A_1') = \tbinom{2a}{a} g_B(A_1'),
\end{align*}
and so $f_B(a,a) = \tbinom{2a}{a}\big[g_B(A_1)+g_B(A_1')\big] = f_B((1,1),A_1) + f_B((1,1),A_1') \le f_B(1,1)$.

If $a\ge d$, then a similar symmetric argument gives $f_B(a,a) \le f_B(n-1,n-1)$.

If $c<a<d$, then $f_B(a,a)=f(a,a)$. Moreover, $c+d\le n$ implies $d-1\le n-c-1$. Since $f(a,a)$ decreases for $a\in[0,\frac{n}{2}]$, and since $f(a,a)=f(n-a,n-a)$, we have
\[ f(a,a) \le f(c+1,c+1) \text{ for every } a\in[c+1,d-1].  \] 
We claim that $f(c+1,c+1)< f_B(1,0)$. To prove this, let $H(c,d) = f_{(c,d)}(I) - f(c+1,c+1)$ where $I=(1,0)$. Then,
\begin{align*}
 H(c,d+1) - H(c,d) &= f(I,(c,d)) - f(I,(c,d+1)) \\
 &= \binom{c+d-1}{c-1}\binom{2n-c-d}{n-c} -  \binom{c+d}{c-1}\binom{2n-c-d-1}{n-c} \\[3pt]
 &= \binom{c+d-1}{c-1}\binom{2n-c-d-1}{n-c}\left[\frac{2n-c-d}{n-d}-\frac{c+d}{d+1}\right] \\[3pt]
 &= \binom{c+d-1}{c-1}\binom{2n-c-d-1}{n-c}\left[\frac{n-c}{n-d}-\frac{c-1}{d+1}\right]>0
\end{align*}
since $n-c>n-d$ and $c-1<d+1$. Thus $H(c,d)>0$ if $H(c,c+2)>0$.

Now let $\Delta(c) = H(c+1,c+3)-H(c,c+2)$ with $2c+4\le n$. If $m=n-c$, then
\begin{align*}
 \Delta(c) &=  f(I,(c,c+2))-f(I,(c+1,c+3)) + f(c+1,c+1) - f(c+2,c+2) \\[5pt]
 &=\text{\footnotesize $\binom{2c+1}{c+2}\binom{2m-2}{m-2} - \binom{2c+4}{c+2}\binom{2m-4}{m-2}
 + \binom{2c+2}{c+1}\binom{2m-2}{m-1} - \binom{2c+3}{c}\binom{2m-4}{m-1}$} \\[5pt]
 &=\binom{2c+1}{c+2}\binom{2m-4}{m-1}\frac{6\,Q(c,m)}{c(c+3)m(m-2)},
\end{align*}
where $Q(c,m) = (c+1)m^2-(c^2+\tfrac{17}{3}c+4)m+c^2+3c$. It can be easily checked that if $n\ge 6$ and $2c+4\le n$, then 
\[ m\ge c+4 > \frac{c^2+\tfrac{17}{3}c+4 + \sqrt{(c^2+\tfrac{17}{3}c+4)^2-4(c+1)(c^2+3c)}}{2(c+1)}, \]
and therefore $Q(c,m)>0$. Thus $\Delta(c)>0$, which implies $H(c,c+2)\ge H(1,3)$.

Finally, for $H(1,3) = f_{(1,3)}(I) - f(2,2)$ we have
\begin{align*}
 H(1,3) &= \binom{2n-1}{n-1} - \binom{2n-4}{n-1} - \binom{4}{2}\binom{2n-4}{n-2} \\
 &= \binom{2n-4}{n-1}\left[\frac{2(2n-1)(2n-3)}{n(n-2)}-1-\frac{6(n-1)}{n-2}\right] \\
 &= \binom{2n-4}{n-1}\frac{n^2-8n+6}{n(n-2)} > 0 \;\text{ if } n\ge 8. 
\end{align*}
In conclusion, if $d\ge c+2$ and $n\ge 8$, we have $H(c,d)>0$ and so $f(c+1,c+1)<f_B(1,0)$. The case $c\ge d+2$ follows by interchanging the roles of $c$ and $d$ (equivalently, reflecting across the main diagonal), yielding $f(d+1,d+1)<f_B(0,1)$.
\end{proof}

\begin{theorem}
If $n\ge 8$ and $B=(c,d)\in\R_2(n)$, then
\[ f_B(A) \le \max\{f_B(1,0),f_B(0,1),f_B(1,1),f_B(n-1,n),f_B(n,n-1),f_B(n-1,n-1)\} \]
for every point $A\in\Grid(n)$ different from $(0,0)$ and $(n,n)$.
\end{theorem}
\begin{proof}
If $c=0$ or $d=0$, the result follows from Lemma~\ref{lem:axes}. Otherwise, by the symmetry \eqref{eq:symmetry} we may assume $c+d\le n$. Lemmas~\ref{lem:A_in_R2}, \ref{lem:subDiagonal}, and \ref{lem:Diagonal} then bound $f_B(A)$ by the right-hand side according to whether the point $A$ is in $\R_2$, $\R_1$, or $\R_0$, respectively.
\end{proof}

\section{Obstructions along the subdiagonals $\R_1$}
\label{sec:subdiagonals}

We now consider the case when the obstruction $B$ is in $\R_1(n) = \{(x,y)\in\Grid(n): |x-y| = 1\}$.

\begin{lemma} \label{lem:R1_A_in_R2}
Let $c\ge 1$ be such that $2c+1\le n$. Further, let $B=(c,c+1)$ and $B'=(c+1,c)$. If $A\in\R_2(n)$, then $f_B(A) < f_B(1,0)$ and $f_{B'}(A) < f_{B'}(0,1)$.
\end{lemma}
\begin{proof}
Let $I=(1,0)$, $B_1=(c-1,c+2)$, and $B_1'=(c+2,c-1)$. The four points $B$, $B'$, $B_1$, $B_1'$ all lie on the antidiagonal $x+y=2c+1$, so any path from $(0,0)$ to $(n,n)$ passes through at most one of them. Since $f(I,B)=\binom{2c}{c-1}\binom{2n-2c-1}{n-c}$, we get
\begin{align*}
  f(I,B_1) &=\binom{2c}{c-2}\binom{2n-2c-1}{n-c+1} = \frac{(c-1)(n-c-1)}{(c+2)(n-c+1)}f(I,B), \\
  f(I,B') &=\binom{2c}{c}\binom{2n-2c-1}{n-c} = \frac{c+1}{c}f(I,B), \\
  f(I,B_1') &=\binom{2c}{c-1}\binom{2n-2c-1}{n-c+1} = \frac{n-c-1}{n-c+1}f(I,B),
\end{align*}
and so
\[ f(I,B_1)+f(I,B')+f(I,B_1') = \left[\frac{(2c+1)(n-c-1)}{(c+2)(n-c+1)} + \frac{c+1}{c} \right] f(I,B) > 2 f(I,B). \]
To verify the inequality, note that
\begin{align*}
  \frac{(2c+1)(n-c-1)}{(c+2)(n-c+1)} + \frac{c+1}{c} -2 
  &= \frac{(c^2+2)n-c^3-3c^2-4c+2}{(c+2)(n-c+1)}.
\end{align*}
The numerator is increasing in $n$ (with coefficient $c^2+2>0$), and at $n=2c+1$ it equals $c^3-2c^2+4$, which is at least $3$ for $c\ge 1$. Hence the entire expression is positive whenever $n\ge 2c+1$.

The above estimate for $f(I,B)$ then implies
\[ f(I)\ge f(I,B)+f(I,B_1)+f(I,B')+f(I,B_1') > 3f(I,B), \] 
which gives $f(I,B)<\tfrac13 f(I)$ and therefore $f_B(I) = f(I) - f(I,B) > \frac23 f(I) = \tfrac{2}{3}\binom{2n-1}{n-1}$.

Finally, by Lemma~\ref{lem:thirdTotal} applied to $A\in\R_2(n)$,
\[ f_B(A)\le f(A) < \frac13 \binom{2n}{n} = \frac{2}{3}\binom{2n-1}{n-1} < f_B(I)=f_B(1,0). \]
The inequality $f_{B'}(A) < f_{B'}(0,1)$ follows by a symmetric argument.
\end{proof}

\begin{lemma} \label{lem:R1_A_in_R1}
Let $c\ge 2$ be such that $2c+1\le n$. Let $B=(c,c+1)$, $B'=(c+1,c)$, $A=(a,a-1)$, and $A'=(a-1,a)$. If $1\le a\le n$ and $n\ge 9$, then
\[ f_B(A')\le f_B(A) \le f_B(1,0) \;\text{ and }\; f_{B'}(A)\le f_{B'}(A') \le f_{B'}(0,1). \]
$($The case $c=1$ will be treated separately in Theorem~\ref{thm:Spoints}.$)$
\end{lemma}
\begin{proof}
We only consider the case for $B$; the argument for $B'$ is symmetric. For $a<c$ or $a\ge c+1$, the inequalities can be shown following the same strategy as in Lemma~\ref{lem:subDiagonal} (Cases 1 and 2 of that proof, respectively). Note that when $a=c+1$, we have $A'=B$, hence $f_B(A')=0$ and the first inequality is trivial.

Suppose now $a=c$. There are more paths passing through $A'=(c-1,c)$ and $B$ than through $A=(c,c-1)$ and $B$, so $f_B(A')\le f_B(A)$. Moreover,
\begin{align*}
 f_B(c+1,c)&= \binom{2c+1}{c}\binom{2n-2c-1}{n-c} \\ 
 &= \frac{4c+2}{c+1}\binom{2c-1}{c}\binom{2n-2c-1}{n-c}
 > 3\,\binom{2c-1}{c}\binom{2n-2c-1}{n-c}
\end{align*}
since $c\ge 2$, and in addition,
\begin{align*}
 f_B(c,c-1) &=  \binom{2c-1}{c}\left[\binom{2n-2c-1}{n-c} +  \binom{2n-2c}{n-c+1}\right] \\
 &= \left[1+\frac{2(n-c)}{n-c+1}\right] \binom{2c-1}{c}\binom{2n-2c-1}{n-c} < 3\,\binom{2c-1}{c}\binom{2n-2c-1}{n-c}.
\end{align*}
Therefore, $f_B(A) = f_B(c,c-1) < f_B(c+1,c) = f(c+1,c)$. We claim that $f(c+1,c)<f_B(1,0)$.

Following the strategy of Lemma~\ref{lem:Diagonal}, let $I=(1,0)$ and consider
\[ H(c) = f_{(c,c+1)}(I) - f(c+1,c) \,\text{ and }\, \Delta(c) = H(c+1)-H(c) \text{ with } 2c+3\le n. \]
If $m=n-c$, then
\begin{align*}
 \Delta(c) &=  f(I,(c,c+1))-f(I,(c+1,c+2)) + f(c+1,c) - f(c+2,c+1) \\[5pt]
 &=\text{\footnotesize $\binom{2c}{c-1}\binom{2m-1}{m} - \binom{2c+2}{c}\binom{2m-3}{m-1}
 + \binom{2c+1}{c}\binom{2m-1}{m} - \binom{2c+3}{c+1}\binom{2m-3}{m-1}$} \\[5pt]
 &=2\binom{2c}{c-1}\binom{2m-3}{m-1}\frac{3cm-3c^2-7c-2}{c(c+2)m}.
\end{align*}
Since $2c+3\le n$, we have $m\ge c+3$, and so $3cm-3c^2-7c-2\ge 2c-2>0$ (using $c\ge 2$). This implies $\Delta(c)>0$ and therefore $H(c)\ge H(2) = f_{(2,3)}(I) - f(3,2)$. Finally,
\begin{align*}
 H(2) &= \binom{2n-1}{n-1} - \binom{4}{1}\binom{2n-5}{n-2} - \binom{5}{2}\binom{2n-5}{n-3} \\
 &= 2\binom{2n-5}{n-2}\frac{n^2-9n+6}{n(n-1)} > 0 \;\text{ if } n\ge 9.
\end{align*}
In conclusion, for $n\ge 9$ we have $H(c)>0$ and so $f(c+1,c)<f_B(1,0)$.
\end{proof}

\begin{lemma} \label{lem:R1_A_in_R0}
Let $c\ge 1$ with $2c+1\le n$, and let $B\in\{(c,c+1),(c+1,c)\}$. If $1\le a \le n-1$, then $f_B(a,a) \le f_B(n-1,n-1)$.
\end{lemma}
\begin{proof}
Without loss of generality, we assume $B=(c,c+1)$. Let 
\begin{alignat*}{3} 
 A&=(a,a), & A_1&=(a-1,a+1),\quad & A_1'&=(a+1,a-1), \\
 J&=(n-1,n-1),\quad & J_1&=(n-2,n), & J_1'&=(n,n-2).
\end{alignat*}
Clearly, $f_B(A) = f_B(A,J_1) + f_B(A,J) + f_B(A,J_1')$, and moreover,
\[ f_B(J)\ge f_B(A_1,J) + f_B(A,J) + f_B(A_1',J). \] 
\smallskip
\noindent{\em Case 1:} $a<c$. We have
\begin{align*}
 f_B(A,J_1) &= \binom{2a}{a}\left[\binom{2n-2a-2}{n-a} - \binom{2c-2a+1}{c-a}\binom{2n-2c-3}{n-c-1}\right], \\
 f_B(A_1, J) &= 2\binom{2a}{a-1}\left[\binom{2n-2a-2}{n-a} - \binom{2c-2a+1}{c-a}\binom{2n-2c-3}{n-c-1}\right].
\end{align*}
Since $\binom{2a}{a} = \frac{a+1}{a}\binom{2a}{a-1}\le 2\binom{2a}{a-1}$, we get $f_B(A,J_1)\le f_B(A_1, J)$. Similarly,
\begin{align*}
 f_B(A,J_1') &= \binom{2a}{a}\left[\binom{2n-2a-2}{n-a} - \binom{2c-2a+1}{c-a}\binom{2n-2c-3}{n-c}\right], \\
 f_B(A_1', J) &= 2\binom{2a}{a-1}\left[\binom{2n-2a-2}{n-a} - \binom{2c-2a+1}{c-a-1}\binom{2n-2c-3}{n-c-1}\right],
\end{align*}
and since $\binom{2c-2a+1}{c-a-1}\binom{2n-2c-3}{n-c-1}<\binom{2c-2a+1}{c-a}\binom{2n-2c-3}{n-c}$, we obtain $f_B(A,J_1')\le f_B(A_1', J)$. Combining the inequalities, we get $f_B(A) \le f_B(J)$.

\smallskip
\noindent{\em Case 2:} $a>c$. A symmetric argument gives $f_B(A) \le f_B(1,1)$. To finish this case, we show that $f_B(1,1)\le f_B(J)$ whenever $n\ge 2c+1$. Since $f(1,1) = f(n-1,n-1)$, we have
\[ f_B(J) - f_B(1,1) = f((1,1),B) - f(J,B), \]
and a direct computation gives
\[ \frac{f(J,B)}{f((1,1),B)} = \frac{(2c+1)(n-c)}{(c+1)(2n-2c-1)}. \]
The inequality $(2c+1)(n-c) \le (c+1)(2n-2c-1)$ simplifies to $n-2c-1\ge 0$, which is exactly our hypothesis. Hence $f_B(1,1)\le f_B(J)$.

\smallskip
\noindent{\em Case 3:} $a=c$. Consider the function $H(c) = f_{(c,c+1)}(J) - f_{(c,c+1)}(c,c)$. Let $m=n-c$ and assume $2c+3\le n$. With routine manipulations, $H(c+1)-H(c)$ simplifies to
\begin{align*}
 H(c+1)-H(c) &= \binom{2c}{c} \binom{2m-1}{m} - 2\binom{2c+3}{c+1} \binom{2m-5}{m-2} \\
 &= 4\binom{2c}{c}\binom{2m-5}{m-2} \left[\frac{(2m-3)(2m-1)}{(m-1)m}-\frac{(2c+1)(2c+3)}{(c+1)(c+2)}\right].
\end{align*}
Since the function $h(x)=\frac{(2x-1)(2x+1)}{x(x+1)}$ is increasing for $x>0$, and since $m-1>c+1$, we conclude $H(c+1)-H(c)>0$. Thus, $H(c)\ge H(1)$ for $c$ with $2c+1\le n$, and
\begin{align*}
 H(1) &= 2\binom{2n-1}{n-1} - 2\binom{3}{1}\binom{2n-5}{n-2} - 2\binom{2n-3}{n-1} \\
 &= 2\binom{2n-5}{n-2}\frac{n-3}{n-1} \ge 0 \;\text{ since } n\ge 3. 
\end{align*}
In conclusion, $f_B(c,c) \le f_B(n-1,n-1)$, which completes the proof.
\end{proof}

\medskip
\begin{theorem} \label{thm:R1}
Let $B=(c,d)\in\R_1(n)$. If $n\ge 9$ and $3<c+d< 2n-3$, then
\[ f_B(A) \le \max\{f_B(1,0),f_B(0,1),f_B(1,1),f_B(n-1,n),f_B(n,n-1),f_B(n-1,n-1)\} \]
for every point $A\in \Grid(n)\backslash\{(0,0),(n,n)\}$.
\end{theorem}
\begin{proof}
By the symmetry \eqref{eq:symmetry}, we may assume $c+d\le n$, so $B=(c,c+1)$ with $c\ge 2$ (the hypothesis $c+d>3$ excludes $c=1$) and $2c+1\le n$. Lemmas~\ref{lem:R1_A_in_R2}, \ref{lem:R1_A_in_R1}, and \ref{lem:R1_A_in_R0} bound $f_B(A)$ by an element of the right-hand side according to whether the point $A$ is in $\R_2$, $\R_1$, or $\R_0$, respectively.
\end{proof}

\begin{remark}
The constraint $n\ge 9$ in Theorem~\ref{thm:R1} appears precisely in Lemma~\ref{lem:R1_A_in_R1}: the inequality $H(2)>0$ requires $n^2-9n+6>0$, which holds (for integer $n$) only when $n\ge 9$. The corresponding results for $B\in\R_2$ and $B\in\R_0$ hold under weaker hypotheses ($n\ge 8$ and $n\ge 5$, respectively). The threshold $n\ge 9$ is also sharp at the level of the main theorem: at $n=8$, four additional maximizer locations appear (namely $(2,3), (3,2), (5,6), (6,5)$, as discussed in Section~\ref{sec:comments}).
\end{remark}

\medskip
We finish the section with the case when the obstruction is any of the points in the set 
\[ \mathcal{S}=\{(1,2),(2,1),(n-1,n-2),(n-2,n-1)\}. \]

\begin{theorem} \label{thm:Spoints}
Let $B=(c,d)\in \mathcal{S}$. If $A\in \Grid(n)\backslash\{(0,0),(n,n)\}$, then $f_B(A)\le f(d,c)$.
\end{theorem}

\begin{proof}
We will only prove the case when $B=(1,2)$; the other three cases can be treated by similar symmetric arguments. Let $A=(a,b)$. The proof proceeds in three parts according to the position of $A$: the easy cases $a\in\{0,1,2\}$; the cases $b\in\{0,1,2\}$ with $a\ge 3$, handled by an injective map $\psi$; and the main case $a,b\ge 3$, which splits into two sub-cases ($a\ge b\ge 3$ and $b>a\ge 3$).

\smallskip
\noindent{\em Case 1:} $a\in\{0,1,2\}$. The bound $f_B(A) \le f(2,1)$ can be verified in each case. For the four points $A\in\{(0,1),(1,0),(1,1),(2,2)\}$, a direct computation gives
\begin{align*}
 f_B(0,1) = f_B(1,0) &= 2\binom{2n-3}{n-2} + \binom{2n-3}{n-3} < 3\binom{2n-3}{n-2} = f(2,1),\\
 f_B(1,1) &= 2\binom{2n-3}{n-2} < f(2,1), \\
 f_B(2,2) &= 3\binom{2n-4}{n-2} < 3\binom{2n-3}{n-2} = f(2,1).
\end{align*}
For $A=(2,1)$, $f_B(A) = f(2,1)$ trivially. For any other $A$ with $a\le 2$, one can check directly that $f(A) \le f(2,1)$, and so $f_B(A) \le f(A) \le f(2,1)$.

\smallskip
\noindent{\em Case 2:} $b\in\{0,1,2\}$ and $a\ge 3$. Every $B$-avoiding path through $A=(a,b)$ passes through exactly one of $(0,3), (2,1), (3,0)$. Since $b\le 2$, the term $f((0,3),A)$ vanishes and
\[ f_B(A) = f((2,1),A) + f((3,0),A). \]
Let $A_1 = (a-1, b+1)$. Then $A$ and $A_1$ are distinct points on the line $x+y = a+b$, so
\[ f(2,1) \ge f((2,1),A) + f((2,1),A_1). \]
A direct computation gives
\[ \frac{f((2,1),A_1)}{f((3,0),A)} = \frac{3(n-b)}{n-a+1} \ge 1, \]
where the inequality follows from $3(n-b) \ge n-a+1$, which is equivalent to $2n+a-3b-1\ge 0$, and holds for $a\ge 3$, $b\le 2$, $n\ge 2$. Hence $f((2,1),A_1) \ge f((3,0),A)$, and we get 
\[ f(2,1) \ge f((2,1),A) + f((3,0),A) = f_B(A). \]

\smallskip
\noindent{\em Case 3:} $a,b\ge 3$. By the same antidiagonal decomposition used in Case~2,
\[ f_B(A) = f((0,3),A) + f((2,1),A) + f((3,0),A), \]
where all three terms are now potentially nonzero (since $a,b\ge 3$). Moreover, if $A_1$ and $A_2$ are points on the line $x+y = a+b$, both different from $A$, then
\[ f(2,1) \ge f((2,1),A) + f((2,1),A_1) + f((2,1),A_2). \]
Our goal is to choose $A_1$ and $A_2$ (depending on $A$) such that
\[ f((2,1),A_1) + f((2,1),A_2) \ge f((0,3),A) + f((3,0),A). \]
Let $\Delta_{A_1} = f((2,1),A_1)- f((3,0),A)$ and $\Delta_{A_2} = f((2,1),A_2)- f((0,3),A)$.

\smallskip
{\em Case 3a:} $a\ge b\ge 3$. Let $A_1=(a-1,b+1)$ and $A_2=(a+1,b-1)$. Then
\begin{align*}
 \Delta_{A_1} &= 3\binom{a+b-3}{a-3} \binom{2n-a-b}{n-a+1} - \binom{a+b-3}{a-3} \binom{2n-a-b}{n-a} \\[1ex]
 &= \binom{a+b - 3}{a-2}\binom{2n-a-b}{n - a} \frac{a-2}{b}\left[\frac{3(n-b)}{n-a+1}-1\right],
\end{align*}
and
\begin{align*}
 \Delta_{A_2} &= 3\binom{a+b-3}{a-1} \binom{2n-a-b}{n-a-1} - \binom{a+b-3}{a} \binom{2n-a-b}{n-a} \\[1ex]
&= \binom{a+b - 3}{a-2}\binom{2n-a-b}{n-a} \frac{b-1}{a-1}\left[\frac{3(n-a)}{n-b+1} - \frac{b-2}{a}\right].
\end{align*}

If $a=b$, then $\frac{3(n-a)}{n-a+1}> 1> \frac{a-2}{a}$ and the expressions inside brackets are all positive. This implies $\Delta_{A_1}+\Delta_{A_2}>0$, and so $f_B(A)<f(2,1)$.

If $a>b\ge 3$, then $\frac{a-2}{b}\ge \frac{b-1}{a-1}$ and $\frac{b-2}{a}<1$, hence
\[ \Delta_{A_1}+\Delta_{A_2}> \binom{a+b - 3}{a-2}\binom{2n-a-b}{n-a} \frac{3(b-1)}{a-1}\left[\frac{n-b}{n-a+1}+\frac{n-a}{n-b+1}-\frac{2}{3}\right]. \]
With $x=n-a+1$ and $y=n-b+1$, the expression inside the bracket equals $h(x,y) = \frac{y-1}{x} + \frac{x-1}{y} - \frac23$. A direct computation gives
\[ h(x,y) = \frac{(x+y)^2 + 2(x-y)^2 - 3(x+y)}{3xy}, \]
so $h(x,y) > 0$ iff $(x+y)^2 + 2(x-y)^2 > 3(x+y)$. For $a>b$ with $n\ge a$, we have $x = n-a+1 \ge 1$ and $y = n-b+1 \ge 2$, so $x+y \ge 3$. If $x+y = 3$, then $(x,y) = (1,2)$ and $2(x-y)^2 = 2 > 0$; if $x+y > 3$, then $(x+y)^2 > 3(x+y)$. In either case the inequality is strict, so $h(x,y) > 0$ and we conclude $\Delta_{A_1}+\Delta_{A_2}>0$.

\smallskip
{\em Case 3b:} $b>a\ge 3$. We now consider $A_1=(a+1,b-1)$ and $A_2=(a+2,b-2)$, and let
\[ \Delta_{A_1} = f((2,1),A_1)- f((3,0),A) \;\text{ and }\; \Delta_{A_2} = f((2,1),A_2)- f((0,3),A). \]
Then
\begin{align*}
 \Delta_{A_1} &= 3\binom{a+b-3}{a-1} \binom{2n-a-b}{n-a-1} - \binom{a+b-3}{a-3} \binom{2n-a-b}{n-a} \\[1ex]
&= \binom{a+b - 3}{a-2}\binom{2n-a-b}{n-a} \left[\frac{3(b-1)(n-a)}{(a-1)(n-b+1)} - \frac{a-2}{b}\right],
\end{align*}
and
\begin{align*}
 \Delta_{A_2} &= 3\binom{a+b-3}{a} \binom{2n-a-b}{n-a-2} - \binom{a+b-3}{a} \binom{2n-a-b}{n-a} \\[1ex]
 &= \binom{a+b - 3}{a}\binom{2n-a-b}{n-a-1} \left[\frac{3(n-a-1)}{n-b+2}-\frac{n-b+1}{n-a}\right].
\end{align*}
Since $a<b$, we have $\frac{3(b-1)(n-a)}{(a-1)(n-b+1)}> 1 >\frac{a-2}{b}$ and so $\Delta_{A_1}>0$. For $\Delta_{A_2}$, we have
\begin{align*}
 \frac{3(n-a-1)}{n-b+2}-\frac{n-b+1}{n-a} &= \frac{3(n-a)(n-a-1) - (n-b+2)(n-b+1)}{(n-b+2)(n-a)} \\
 &\ge \frac{3(n-a)(n-a-1) - (n-a+1)(n-a)}{(n-b+2)(n-a)} \\
 &= \frac{2(n-a-2)}{n-b+2} \ge 0 \quad\text{if } a\le n-2,
\end{align*}
where the inequality uses $(n-b+2)(n-b+1)\le(n-a+1)(n-a)$, which follows from $a<b$. Thus $\Delta_{A_2}\ge 0$ when $a\le n-2$, and the bound $f_B(A)\le f(2,1)$ follows.

The remaining case is $a=n-1$ and $b=n$, where the choice $A_2 = (a+2,b-2) = (n+1,n-2)$ falls outside $\Grid(n)$ and the above argument does not apply. In this case, $f((2,1), A_2) = 0$, so
\[ \Delta_{A_2} = f((2,1),A_2) - f((0,3),A) = -f((0,3),A) = -\binom{2n-4}{n-1}. \]
A direct computation then gives
\begin{align*}
 \Delta_{A_1}+\Delta_{A_2} &= \binom{2n-4}{n-1} \left[\frac{3(n-1)}{n-2} - \frac{n-3}{n}\right] -\binom{2n-4}{n-1}\\
 &= \binom{2n-4}{n-1} \left[\frac{3(n-1)}{n-2} - \frac{n-3}{n} -1\right] \\
 &= \binom{2n-4}{n-1} \frac{n^2+4n-6}{n(n-2)},
\end{align*}
which is positive since $n\ge 3$.
\end{proof}

\section{Obstructions along $\R_0$} 
\label{sec:diagonal}

In this section we focus on obstructions along the diagonal region 
\[ \R_0(n) = \{(x,y)\in\Grid(n): x=y, \text{ and } 0<x<n\}. \]

\begin{theorem}\label{thm:R0}
Let $B=(c,c)\in\R_0(n)$ with $n\ge 5$. For every $A\in\Grid(n)\backslash\{(0,0),(n,n)\}$, we have
\begin{equation*}
 f_B(A) \le 
 \begin{cases}
 f_B(1,1) & \text{if $c<n\le 2c$}, \\
 f_B(n-1,n-1) & \text{if $2c \le n$.}
 \end{cases}
\end{equation*}
\end{theorem}
\begin{proof}
Let $A=(a,b)\in\Grid(n)$. The proof has two parts: the off-diagonal case ($a\ne b$), where we reduce $f_B(A) \le f_B(1,1)$ to a single algebraic inequality; and the on-diagonal case ($a=b\ne c$), handled by an antidiagonal decomposition. The case $2c\le n$ follows from the case $c<n\le 2c$ by the reflection $(x,y)\mapsto (n-x,n-y)$.

Suppose first that $a\ne b$. Every $B$-avoiding path through $A$ takes either $\N$ or $\E$ as its first step, reaching $(0,1)$ or $(1,0)$ respectively. Reflection about the line $y=x$ (which fixes $B$) gives $f_B((1,0),A) = f_B((0,1),A')$ where $A'=(b,a)$. Since $a\ne b$, the points $A$ and $A'$ are distinct and lie on the antidiagonal $x+y=a+b$, so
\[ f_B(A) = f_B((0,1),A) + f_B((1,0),A) = f_B((0,1),A) + f_B((0,1),A') \le f_B(0,1). \]
It remains to show $f_B(0,1)\le f_B(1,1)$. Decomposing $f_B(1,1)$ by the first step and $f_B(0,1)$ by the second step, we get $f_B(1,1) - f_B(0,1) = f_B((1,0),(1,1))-f_B(0,2)$. We denote this difference by $\Delta$ and proceed to show $\Delta\ge 0$. Routine manipulations give
\begin{equation*}
 \Delta = \frac{1}{n}\binom{2n-2}{n-1} - \binom{2c-2}{c}\binom{2n-2c}{n-c-1}\frac{n-c+1}{(c-1)(n-c)}.
\end{equation*}

If $c<n\le 2c$ and $n\ge 5$, then we must have $c\ge 3$. For $c\in\{3,4\}$, direct computation gives $\Delta \ge 2$ in all six admissible cases (with $\Delta=2$ when $c=3$ and $n\in\{5,6\}$). For $c\ge 5$, we think of $(c-2,c)$ as a point in $\Grid(n-1)$ and use Lemma~\ref{lem:thirdTotal} to conclude
\[ \binom{2c-2}{c}\binom{2n-2c}{n-c-1} < \frac13 \binom{2n-2}{n-1}. \]
Therefore,
\begin{equation*}
 \Delta >  \binom{2n-2}{n-1}\left[\frac{1}{n} - \frac{n-c+1}{3(c-1)(n-c)}\right] = \binom{2n-2}{n-1}\frac{Q(c,n)}{3n(c-1)(n-c)},
\end{equation*}
where $Q(c,n) = -3c(c-1)+4(c-1)n-n^2$. Treating $Q$ as a quadratic polynomial in $n$, its roots are $(2c-2)\pm\sqrt{(c-1)(c-4)}$, and a short calculation shows that, for $c\ge 5$, the hypothesis range $c+1\le n\le 2c$ lies between these roots. Hence $Q(c,n)\ge 0$, and so $\Delta>0$.

In conclusion, $f_B(A)\le f_B(1,1)$ whenever $A$ is not on the diagonal $y=x$. It remains to examine the case when $A=(a,a)$ with $1\le a \le n-1$ and $a\not=c$.

Suppose first that $a<c$. Every path from $A$ to $(n,n)$ crosses the antidiagonal $x+y=2c$ at exactly one point, and the points of this antidiagonal that are reachable from $A$ and lie in $\Grid(n)$ are $B$ together with $B_k=(c-k,c+k)$ and $B_k'=(c+k,c-k)$ for $1\le k\le \min\{c-a,n-c\}$. Hence, with $K=\min\{c-a,n-c\}$,
\[ f_B(A) = \sum_{k=1}^{K} \big(f(A,B_k)+f(A,B_k')\big). \]
If $a=1$, then $A=(1,1)$ and the bound $f_B(A)\le f_B(1,1)$ is trivially true, so we may assume $a>1$. Let $J=(1,1)$, $A_1=(a-1,a+1)$, and $A_1'=(a+1,a-1)$. Clearly,
\[ f(J,B_k) - f(A,B_k) \ge f(J,A_1,B_k) + f(J,A,B_k) + f(J,A_1',B_k) - f(A,B_k). \]
Let $\Delta_k = f(J,A_1,B_k) + f(J,A,B_k) + f(J,A_1',B_k) - f(A,B_k)$. If $k<c-a$, then
\begin{align*}
 \Delta_k = 2\binom{2a-2}{a}\binom{2c-2a}{c-k-a}\binom{2n-2c}{n-c+k}
 \frac{(c-a)^2+3k^2-1}{(c-k-a+1)(c+k-a+1)} > 0.
\end{align*}
On the other hand, if $k=c-a$ (which implies $k=K$), then $f(J,A_1',B_k)=0$ and
\begin{align*}
 \Delta_{c-a} = 2\binom{2a-2}{a}\binom{2n-2c}{n-a}(2c-2a-1) > 0.
\end{align*}
In conclusion, $f(J,B_k) - f(A,B_k) \ge \Delta_k>0$ for every $k\in\{1,\dots,K\}$.

Similarly, using $B_k'$ as reference point, one gets $f(J,B_k') - f(A,B_k')\ge 0$ and thus
\[ f_B(J) - f_B(A) \ge  \sum_{k=1}^{K} \big(f(J,B_k) - f(A,B_k) + f(J,B_k') - f(A,B_k')\big) > 0. \]

The case $c<a\le n-1$ reduces to the previous one via the reflection $\sigma(x,y)=(n-x,n-y)$. Setting $\tilde B=\sigma(B)=(n-c,n-c)$ and $\tilde A=\sigma(A)=(n-a,n-a)$, identity~\eqref{eq:symmetry} gives $f_B(A)=f_{\tilde B}(\tilde A)$. Since $n\le 2c$, we have $\tilde a<\tilde c$, so the argument above yields $f_{\tilde B}(\tilde A)\le f_{\tilde B}(1,1) = f_B(n-1,n-1)$. Finally, $f_B(n-1,n-1)\le f_B(1,1)$ follows from $f(1,1)=f(n-1,n-1)$ together with the inequality $f((1,1),B)\le f((n-1,n-1),B)$.
\end{proof}

\section{Classification and small cases}
\label{sec:comments}

The results of Sections~\ref{sec:2unitsAway}--\ref{sec:diagonal}, together with Lemma~\ref{lem:axes}, can be summarized as follows. Let $n\ge 9$ and let $B\in\Grid(n)\backslash\{(0,0),(n,n)\}$ be an obstruction. Let
\begin{equation}
\label{eq:maximizers}
\begin{gathered}
 \mathcal{S} = \{(1,2),(2,1),(n-1,n-2),(n-2,n-1)\}, \text{ and} \\
 \mathcal{C} = \{(1,0),(0,1),(1,1),(n-1,n),(n,n-1),(n-1,n-1)\}. 
\end{gathered}
\end{equation}
If $B\in\mathcal{S}$, then by Theorem~\ref{thm:Spoints} the maximum of $f_B(A)$ over $A\in\Grid(n)\backslash\{(0,0),(n,n)\}$ is attained at the reflection of $B$ across the diagonal $y=x$, which is the other point of $\mathcal{S}$ in the same diagonal half of the grid. For every other obstruction $B$, the maximum is attained at one of the six corner-adjacent points of $\mathcal{C}$: this follows from Lemma~\ref{lem:axes} when $B$ is on the boundary of $\Grid(n)$, from the main theorem of Section~\ref{sec:2unitsAway} when $B\in\R_2(n)$, from Theorem~\ref{thm:R1} when $B\in\R_1(n)\backslash\mathcal{S}$, and from Theorem~\ref{thm:R0} when $B\in\R_0(n)$. Thus the maximum is always attained at one of the ten points of $\mathcal{S}\cup\mathcal{C}$, as claimed in the introduction.

Figures~\ref{fig:case_n56}--\ref{fig:case_n=8} show, for each $5\le n\le 8$ and each obstruction $B\in\Grid(n)\backslash\{(0,0),(n,n)\}$, the location of the point(s) of maximum traffic. A dot at $B$ is colored according to its maximizer: blue for $(1,1)$, cyan for $(n-1,n-1)$, green for a tie between the two, and yellow for any other maximizer (whose location is given by the label below~$B$). Gray dots mark obstructions whose maximizer is at one of $(0,1)$, $(1,0)$, $(n-1,n)$, or $(n,n-1)$. In the $n=8$ picture, red marks the two obstructions $B=(3,5)$ and $B=(5,3)$ on the antidiagonal $x+y=8$ for which the maximum is tied between two points of $\mathcal{C}$ other than $(1,1)$ and $(n-1,n-1)$.

\begin{figure}[ht]
\begin{tikzpicture}[scale=0.95]
\begin{scope}
\draw[lightgray] (0,0) grid (5,5);
\foreach \x/\y in {0/0,5/5}{\draw[fill] (\x,\y) circle (0.13);}
\foreach \x in {1,...,5}{\draw[fill=gray!50] (\x,0) circle (0.1);}
\foreach \x in {0,...,4}{\draw[fill=gray!50] (\x,5) circle (0.1);}
\foreach \y in {1,...,4}{\foreach \x in {0,...,5}{\draw[fill=gray!50] (\x,\y) circle (0.1);}}
\foreach \x/\y in {3/3,4/4,0/2,0/3,0/4,2/0,3/0,4/0}{
\draw[fill=blue!60] (\x,\y) circle (0.1);
\node[below=1pt] at (\x,\y) {\tiny $(1,1)$};}
\foreach \x/\y in {1/1,2/2,1/5,2/5,3/5,5/1,5/2,5/3}{
\draw[fill=cyan!60] (\x,\y) circle (0.1);
\node[below=1pt] at (\x,\y) {\tiny $(4,4)$};}
\foreach \x/\y in {1/2,2/1,2/3,3/2,1/3,3/1,2/4,4/2,3/4,4/3}{
\draw[fill=sMax] (\x,\y) circle (0.1);}
\foreach \x/\y in {0/1,1/0,1/2,2/1,2/3,3/2,3/4,4/3,4/5,5/4}{
\node[below=1pt] at (\x,\y) {\tiny $(\y,\x)$};}
\foreach \x/\y in {1/3,3/1}{
\node[below=1pt] at (\x,\y) {\tiny $(2,2)$};}
\foreach \x/\y in {2/4,4/2}{
\node[below=1pt] at (\x,\y) {\tiny $(3,3)$};}
\foreach \x/\y in {0/5,1/4,4/1,5/0}{
\draw[fill=2Max] (\x,\y) circle (0.1);
\node[below=1pt] at (\x,\y) {\tiny $(1,1)$};
\node[below=9pt] at (\x,\y) {\tiny $(4,4)$};}
\end{scope}
\begin{scope}[xshift=200,yshift=-14]
\draw[lightgray] (0,0) grid (6,6);
\foreach \x/\y in {0/0,6/6}{\draw[fill] (\x,\y) circle (0.13);}
\foreach \x in {1,...,6}{\draw[fill=gray!50] (\x,0) circle (0.1);}
\foreach \x in {0,...,5}{\draw[fill=gray!50] (\x,6) circle (0.1);}
\foreach \y in {1,...,5}{\foreach \x in {0,...,6}{\draw[fill=gray!50] (\x,\y) circle (0.1);}}
\foreach \x/\y in {4/4,5/5,0/2,0/3,0/4,0/5,2/0,3/0,4/0,5/0,1/4,4/1}{
\draw[fill=blue!60] (\x,\y) circle (0.1);
\node[below=1pt] at (\x,\y) {\tiny $(1,1)$};}
\foreach \x/\y in {1/1,2/2,1/6,2/6,3/6,4/6,6/1,6/2,6/3,6/4,2/5,5/2}{
\draw[fill=cyan!60] (\x,\y) circle (0.1);
\node[below=1pt] at (\x,\y) {\tiny $(5,5)$};}
\foreach \x/\y in {1/2,2/1,2/3,3/2,1/3,3/1,2/4,4/2,3/4,4/3,3/5,5/3,4/5,5/4}{
\draw[fill=sMax] (\x,\y) circle (0.1);}
\foreach \x/\y in {0/1,1/0,1/2,2/1,2/3,3/2,3/4,4/3,4/5,5/4,5/6,6/5}{
\node[below=1pt] at (\x,\y) {\tiny $(\y,\x)$};}
\foreach \x/\y in {1/3,3/1}{
\node[below=1pt] at (\x,\y) {\tiny $(2,2)$};}
\foreach \x/\y in {3/5,5/3}{
\node[below=1pt] at (\x,\y) {\tiny $(4,4)$};}
\foreach \x/\y in {2/4,4/2}{
\node[below=1pt] at (\x,\y) {\tiny $(3,3)$};}
\foreach \x/\y in {0/6,1/5,3/3,5/1,6/0}{
\draw[fill=2Max] (\x,\y) circle (0.1);
\node[below=1pt] at (\x,\y) {\tiny $(1,1)$};
\node[below=9pt] at (\x,\y) {\tiny $(5,5)$};}
\end{scope}
\end{tikzpicture}
\caption{Cases $n=5$ and $n=6$.}
\label{fig:case_n56}
\end{figure}
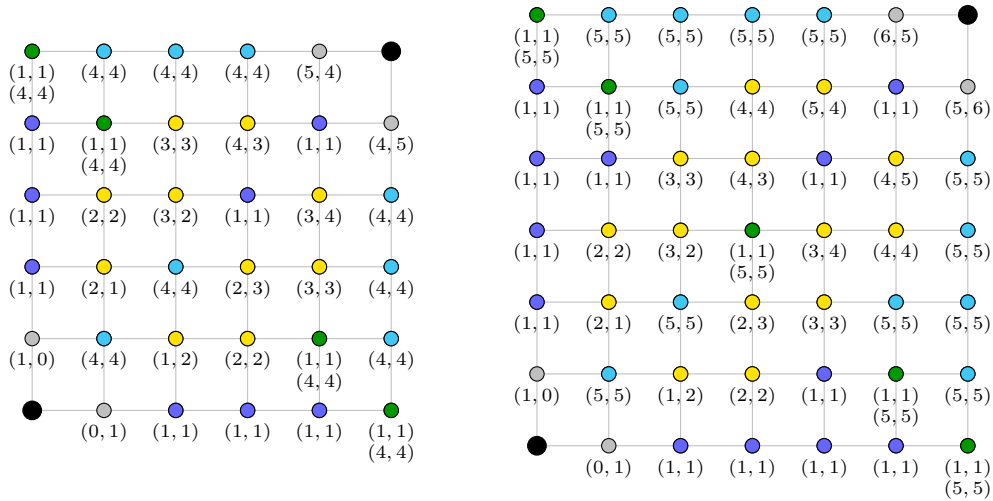

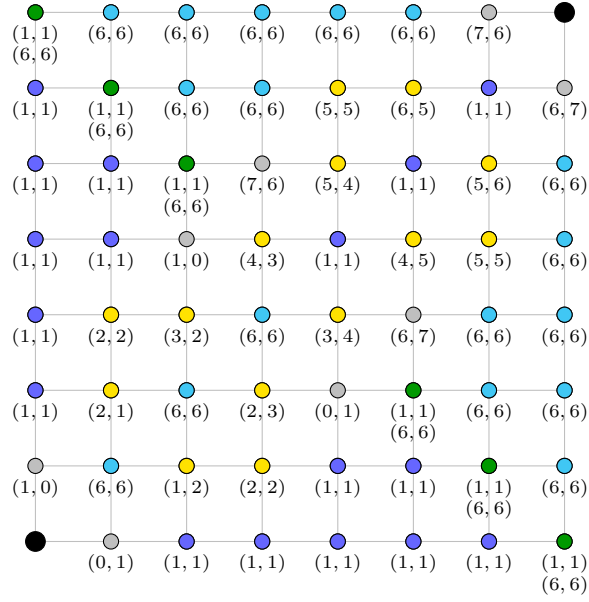
\begin{figure}[ht]
\begin{tikzpicture}[scale=1]
\draw[lightgray] (0,0) grid (7,7);
\foreach \x/\y in {0/0,7/7}{\draw[fill] (\x,\y) circle (0.13);}
\foreach \x in {1,...,7}{\draw[fill=gray!50] (\x,0) circle (0.1);}
\foreach \x in {0,...,6}{\draw[fill=gray!50] (\x,7) circle (0.1);}
\foreach \y in {1,...,6}{\foreach \x in {0,...,7}{\draw[fill=gray!50] (\x,\y) circle (0.1);}}
\foreach \x/\y in {4/4,5/5,6/6,0/2,0/3,0/4,0/5,0/6,2/0,3/0,4/0,5/0,6/0,1/4,1/5,4/1,5/1}{
\draw[fill=blue!60] (\x,\y) circle (0.1);
\node[below=1pt] at (\x,\y) {\tiny $(1,1)$};}
\foreach \x/\y in {1/1,2/2,3/3,1/7,2/7,3/7,4/7,5/7,7/1,7/2,7/3,7/4,7/5,2/6,6/2,3/6,6/3}{
\draw[fill=cyan!60] (\x,\y) circle (0.1);
\node[below=1pt] at (\x,\y) {\tiny $(6,6)$};}
\foreach \x/\y in {1/2,2/1,2/3,3/2,1/3,3/1,3/4,4/3,4/5,5/4,4/6,6/4,5/6,6/5}{
\draw[fill=sMax] (\x,\y) circle (0.1);}
\foreach \x/\y in {0/1,1/0,1/2,2/1,2/3,3/2,3/4,4/3,4/5,5/4,5/6,6/5,6/7,7/6}{
\node[below=1pt] at (\x,\y) {\tiny $(\y,\x)$};}
\foreach \x/\y in {1/3,3/1}{
\node[below=1pt] at (\x,\y) {\tiny $(2,2)$};}
\foreach \x/\y in {4/6,6/4}{
\node[below=1pt] at (\x,\y) {\tiny $(5,5)$};}
\foreach \x/\y in {2/4}{
\node[below=1pt] at (\x,\y) {\tiny $(1,0)$};}
\foreach \x/\y in {4/2}{
\node[below=1pt] at (\x,\y) {\tiny $(0,1)$};}
\foreach \x/\y in {3/5}{
\node[below=1pt] at (\x,\y) {\tiny $(7,6)$};}
\foreach \x/\y in {5/3}{
\node[below=1pt] at (\x,\y) {\tiny $(6,7)$};}
\foreach \x/\y in {0/7,1/6,2/5,5/2,6/1,7/0}{
\draw[fill=2Max] (\x,\y) circle (0.1);
\node[below=1pt] at (\x,\y) {\tiny $(1,1)$};
\node[below=9pt] at (\x,\y) {\tiny $(6,6)$};}
\end{tikzpicture}
\caption{Case $n=7$.}
\label{fig:case_n=7}
\end{figure}

\begin{figure}[ht]
\begin{tikzpicture}[scale=1]
\draw[lightgray] (0,0) grid (8,8);
\foreach \x/\y in {0/0,8/8}{\draw[fill] (\x,\y) circle (0.13);}
\foreach \x in {1,...,8}{\draw[fill=gray!50] (\x,0) circle (0.1);}
\foreach \x in {0,...,7}{\draw[fill=gray!50] (\x,8) circle (0.1);}
\foreach \y in {1,...,7}{\foreach \x in {0,...,8}{\draw[fill=gray!50] (\x,\y) circle (0.1);}}
\foreach \x/\y in {5/5,6/6,7/7,0/2,0/3,0/4,0/5,0/6,0/7,2/0,3/0,4/0,5/0,6/0,7/0,1/4,1/5,1/6,4/1,5/1,6/1,2/5,5/2}{
\draw[fill=blue!60] (\x,\y) circle (0.1);
\node[below=1pt] at (\x,\y) {\tiny $(1,1)$};}
\foreach \x/\y in {1/1,2/2,3/3,1/8,2/8,3/8,4/8,5/8,6/8,8/1,8/2,8/3,8/4,8/5,8/6,2/7,7/2,3/6,6/3,3/7,7/3,4/7,7/4}{
\draw[fill=cyan!60] (\x,\y) circle (0.1);
\node[below=1pt] at (\x,\y) {\tiny $(7,7)$};}
\foreach \x/\y in {1/2,2/1,2/3,3/2,5/6,6/5,6/7,7/6}{
\draw[fill=sMax] (\x,\y) circle (0.1);}
\foreach \x/\y in {0/1,1/0,1/2,2/1,2/3,3/2,5/6,6/5,6/7,7/6,7/8,8/7}{
\node[below=1pt] at (\x,\y) {\tiny $(\y,\x)$};}
\foreach \x/\y in {1/3,2/4,3/4}{
\node[below=1pt] at (\x,\y) {\tiny $(1,0)$};}
\foreach \x/\y in {3/1,4/2,4/3}{
\node[below=1pt] at (\x,\y) {\tiny $(0,1)$};}
\foreach \x/\y in {4/5,4/6,5/7}{
\node[below=1pt] at (\x,\y) {\tiny $(8,7)$};}
\foreach \x/\y in {5/4,6/4,7/5}{
\node[below=1pt] at (\x,\y) {\tiny $(7,8)$};}
\foreach \x/\y in {0/8,1/7,2/6,6/2,7/1,8/0,4/4}{
\draw[fill=2Max] (\x,\y) circle (0.1);
\node[below=1pt] at (\x,\y) {\tiny $(1,1)$};
\node[below=9pt] at (\x,\y) {\tiny $(7,7)$};}
\foreach \x/\y in {3/5}{
\draw[fill=red!80!black] (\x,\y) circle (0.1);
\node[below=1pt] at (\x,\y) {\tiny $(1,0)$};
\node[below=9pt] at (\x,\y) {\tiny $(8,7)$};}
\foreach \x/\y in {5/3}{
\draw[fill=red!80!black] (\x,\y) circle (0.1);
\node[below=1pt] at (\x,\y) {\tiny $(0,1)$};
\node[below=9pt] at (\x,\y) {\tiny $(7,8)$};}
\end{tikzpicture}
\caption{Case $n=8$.}
\label{fig:case_n=8}
\end{figure}
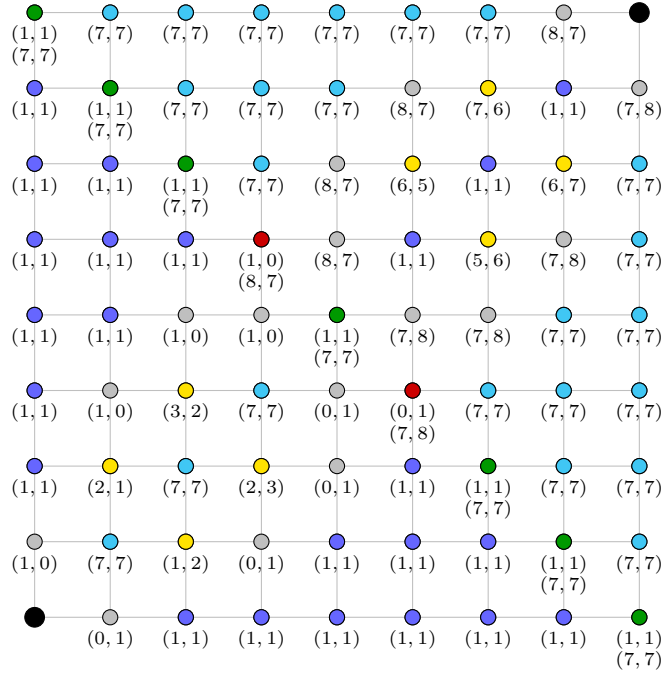

For $n\ge 9$, the picture stabilizes: the only yellow dots are at the four obstructions in $\mathcal{S}$, and their maximizers are also in $\mathcal{S}$ by Theorem~\ref{thm:Spoints}. Every other obstruction produces a blue, cyan, green, gray or red dot. The pictures for $5\le n\le 8$ show how this stable configuration is reached, with extra yellow dots near the diagonal $y=x$ at $n=5,6,7$ and the four additional locations $(2,3)$, $(3,2)$, $(5,6)$, $(6,5)$ at $n=8$. Figure~\ref{fig:n300} displays the distribution of maximizers for $n=300$.

\begin{figure}[ht]
\includegraphics[scale=0.28]{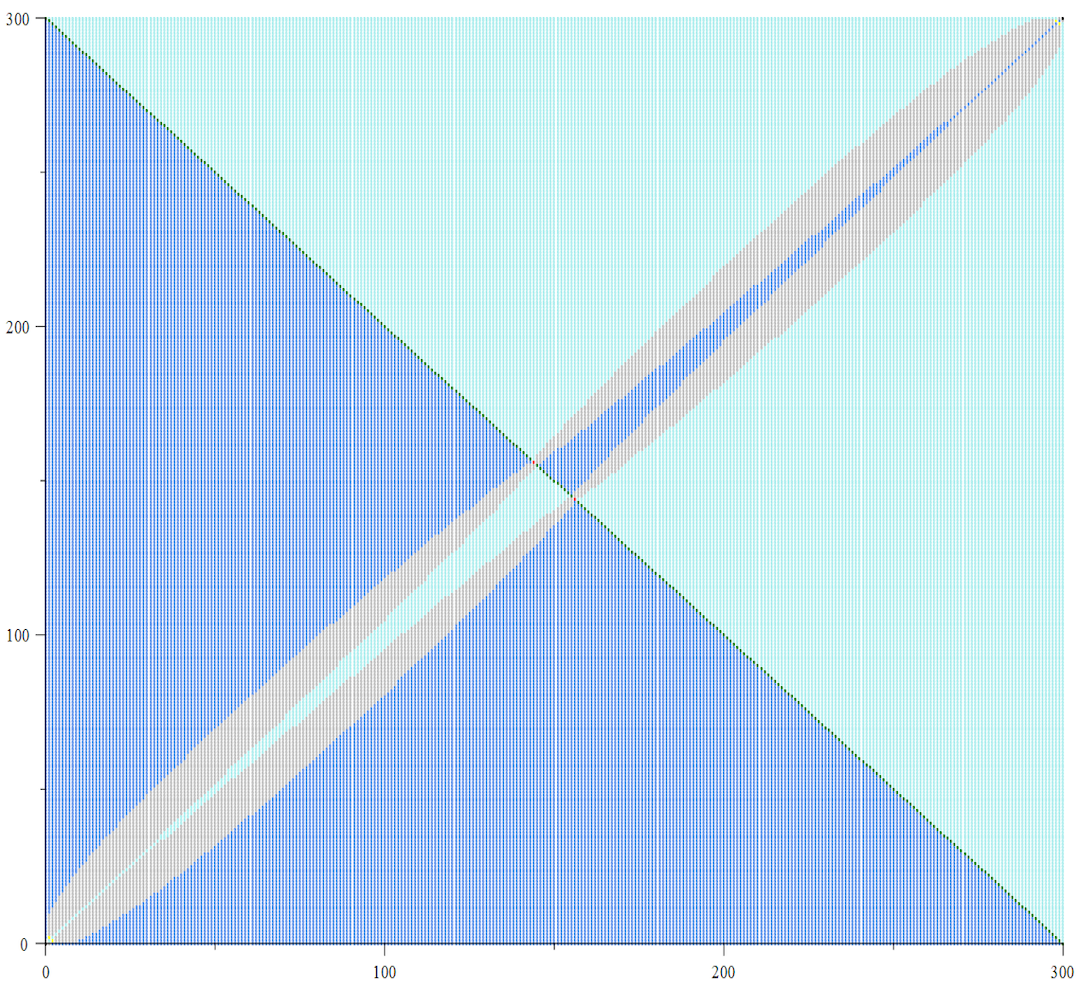}
\caption{Color-coded distribution of maximizers for $n=300$.}
\label{fig:n300}
\end{figure}

An interesting anomaly that caught our attention is that, along the antidiagonal $x+y=n$, the maximizer is \emph{not} always at $(1,1)$ and $(n-1,n-1)$, and the maximum migrates (irregularly) to a pair of boundary points of $\mathcal{C}$.
We discuss this in the next section.

\section{Antidiagonal anomaly and open problems}
\label{sec:antidiagonal}

Throughout this section $B$ denotes an obstruction on the antidiagonal $x+y=n$. 

\begin{proposition}\label{prop:antidiag-ties}
Let $n\ge 2$ and $B=(a,n-a)$ with $1\le a\le n-1$. Then
\[ f_B(x,y) = f_B(n-y,\,n-x) \;\text{ for every } (x,y)\in\Grid(n)\setminus\{(0,0),(n,n)\}. \]
In particular, $f_B(1,1)=f_B(n-1,n-1)$, $f_B(1,0)=f_B(n,n-1)$, and $f_B(0,1)=f_B(n-1,n)$. Moreover, $f_B(0,1)<f_B(1,0)$ if $2a<n$, and $f_B(1,0)<f_B(0,1)$ if $2a>n$.
\end{proposition}

\begin{proof}
Composing the two identities in \eqref{eq:symmetry} gives $f_B(A)=f_{\widetilde{B'}}(\widetilde{A'})$, and $(x,y)\mapsto(n-y,n-x)$ is the reflection about the antidiagonal $x+y=n$. Since $B$ lies on that line, we have $\widetilde{B'}=B$ and so $f_B(x,y) = f_B(n-y,\,n-x)$. For the inequalities, since $f(1,0)=f(0,1)$, we only need to compare $f\big((1,0),B\big)=\binom{n}{a}\binom{n-1}{a-1}$ and $f\big((0,1),B\big)=\binom{n}{a}\binom{n-1}{a}$. If $2a<n$, then $\binom{n-1}{a-1}<\binom{n-1}{a}$ and we get $f_B(1,0)>f_B(0,1)$. Similarly, if $2a>n$, then $f_B(1,0)<f_B(0,1)$.
\end{proof}

In other words, the six points of $\mathcal{C}$ collapse to three contenders. By symmetry, we may assume $2a<n$, reducing the maximizer debate to a comparison between $f_B(1,0)$ and $f_B(1,1)$.

\begin{proposition}\label{prop:antidiag-criterion}
Let $B=(a,n-a)$ with $1\le a\le n-1$. Then 
\[ f_B(1,0)-f_B(1,1) = G(n,a)-D(n), \] 
where $D(n) = \frac{1}{n}\binom{2n-2}{n-1}$ and $G(n,a) = \frac{n-2a+1}{n-a}\binom{n}{a}\binom{n-2}{a-1}$. 
\end{proposition}
\begin{proof}
By definition, $f_B(A)=f(A)-f(A,B)$, so
\[ f_B(1,0)-f_B(1,1) = \big[f\big((1,1),B\big)-f\big((1,0),B\big)\big] - \big[f(1,1)-f(1,0)\big]. \]
Since $f(1,1)=2\binom{2n-2}{n-1}$ and $f(1,0)=\binom{2n-1}{n-1}=\frac{2n-1}{n}\binom{2n-2}{n-1}$, the second bracket is precisely $D(n)$. Furthermore, $f\big((1,1),B\big)=2\binom{n}{a}\binom{n-2}{a-1}$ and $f\big((1,0),B\big)=\binom{n}{a}\binom{n-1}{a-1}$, so by $\binom{n-1}{a-1}=\frac{n-1}{n-a}\binom{n-2}{a-1}$, the first bracket equals $\big[2-\frac{n-1}{n-a}\big] \binom{n}{a}\binom{n-2}{a-1}=G(n,a)$.
\end{proof}

Writing 
\[ R_n(a)=\frac{G(n,a)}{D(n)}, \] 
we conclude that an obstruction $(a,n-a)$ with $2a<n$ satisfies $f_B(1,0)>f_B(1,1)$ precisely when $R_n(a)>1$. Setting $\rho(n)=\max\limits_{1\le a<n/2} R_n(a)$, and recalling that for $n\ge 9$ the maximum is attained on $\mathcal{C}$, we see that some antidiagonal obstruction has its maximum at the boundary points precisely when $\rho(n)>1$.

\begin{lemma}\label{lem:monotone}
For $2\le a<n/2$,
\[ \frac{R_n(a-1)}{R_n(a)} \;=\; \frac{a\,(a-1)\,(n-2a+3)}{(n-a+1)^2\,(n-2a+1)}, \]
and this quotient is less than $1$ whenever $n-2a\ge\sqrt{n}$. Consequently, if $a^*$ denotes the largest index with $n-2a^*\ge\sqrt{n}$, then 
\[ \rho(n) \;=\; \max_{a^*\le a<n/2} R_n(a). \]
\end{lemma}

\begin{proof}
Since $D(n)$ does not depend on $a$, we have $R_n(a-1)/R_n(a)=G(n,a-1)/G(n,a)$. Using $\binom{n}{a-1}=\frac{a}{n-a+1}\binom{n}{a}$ and $\binom{n-2}{a-2}=\frac{a-1}{n-a}\binom{n-2}{a-1}$, the definition of $G$ gives
\[ \frac{G(n,a-1)}{G(n,a)}
   = \frac{n-2a+3}{n-2a+1}\cdot\frac{n-a}{n-a+1}\cdot\frac{a}{n-a+1}\cdot\frac{a-1}{n-a}
   = \frac{a\,(a-1)\,(n-2a+3)}{(n-a+1)^2\,(n-2a+1)}. \]
For the inequality, let $u=n-2a+1$ and observe that $R_n(a-1)/R_n(a)<1$ if and only if $(u+a)^2\,u>a(a-1)(u+2)$. Now, $n-2a\ge\sqrt n$ implies $u > \sqrt{n}+1$, and so $u^2 > n > a$. Therefore,
\[ (u+a)^2 u = u^3 + 2au^2 + a^2u > 2a^2 + a^2 u = a^2(u+2) > a(a-1)(u+2). \]
The last assertion follows, since $R_n(a-1)<R_n(a)$ for every $a\le a^*$.
\end{proof}

Evaluating this criterion in rational arithmetic gives the following. The inequality $\rho(n)>1$ holds for every $n$ with $5\le n\le 375$. For $n\le 7$, however, it does not yet produce a boundary maximizer: the global maximum for near-central antidiagonal obstructions is then attained at interior points, as recorded in Section~\ref{sec:comments}. The actual anomaly first appears at $n=8$, where the obstructions $(3,5)$ and $(5,3)$ have their maxima at $\{(1,0),(8,7)\}$ and $\{(0,1),(7,8)\}$, and it occurs for every $8\le n\le 375$. In the window $376\le n\le 495$, it occurs for $82$ of the $120$ values of $n$; the $38$ exceptions are:
\[ 376,\ 378, \qquad 423,\ 425,\dots,\ 463 \text{ (odd)},\qquad 466,\ 468, \dots,\ 494 \text{ (even)}. \]
It appears that the last value of $n$ for which the anomaly occurs is $n=495$. For $496\le n\le 2000$ we have verified that $\rho(n)<1$. 

\begin{conjecture}\label{conj:496}
For every $n\ge 496$ and every obstruction $B$ on the antidiagonal $x+y=n$, the maximum of $f_B$ is attained at $(1,1)$ and $(n-1,n-1)$. Equivalently, $\rho(n)<1$ for all $n\ge 496$.
\end{conjecture}

\subsection*{Open problems}

Several natural questions remain. Beyond Conjecture~\ref{conj:496}, the most intriguing to us is understanding the curved boundaries separating the colored regions in Figure~\ref{fig:n300}. 

Further natural extensions include obstructions on rectangular grids, and larger obstructions consisting of entire segments or more general convex regions. In each case, the basic question is the same: which points carry the maximum traffic?


\end{document}